\documentclass[a4paper,11pt,leqno]{amsart}
\usepackage{cite}
\usepackage[english]{babel}
\usepackage[utf8]{inputenc}
\usepackage[T1]{fontenc}
\usepackage{amsthm}
\usepackage{amssymb}
\usepackage{amsmath}

\usepackage{amsfonts}
\usepackage{graphicx}
\usepackage[all]{xy}
\newtheorem{Thm}{Theorem}[section]

\newtheorem{Cor}[Thm]{Corollary}
\newtheorem{Apulause}[Thm]{Lemma}

\theoremstyle{definition}

\renewcommand{\r}{|}
\newcommand{\R}{\mathbb{R}}

\newcommand{\N}{\mathbb{N}}
\newcommand{\Z}{\mathbb{Z}}

\newcommand{\es}{\varnothing} % tyhjä joukko

\newcommand{\Deck}{\mathrm{Deck}}

\renewcommand{\Cup}{\bigcup}
\renewcommand{\Cap}{\bigcap}

\newcommand{\ol}{\overline}
\newcommand{\ra}{\rightarrow}

\title[Non-manifold monodromy spaces]{Non-manifold monodromy spaces of branched coverings between manifolds}
\author{Martina Aaltonen}
\date{\today}
 
\begin{document}
\maketitle
\begin{flushright}
\end{flushright}
\begin{abstract}By a construction of Berstein and Edmonds every proper branched cover $f$ between manifolds is a factor of a branched covering orbit map from a locally connected and locally compact Hausdorff space called the monodromy space of $f$ to the target manifold. For proper branched covers between $2$-manifolds the monodromy space is known to be a manifold. We show that this does not generalize to dimension $3$ by constructing a self-map of the 3-sphere for which the monodromy space is not a locally contractible space. \end{abstract}
\section{Introduction}

 A map $f \colon X \to Y$ between topological spaces is a \textit{branched covering}, if $f$ is open, continuous and discrete map.The \textit{branch set} $B_f \subset X$ of $f$ is the set of points in $X$ for which $f$ fails to be a local homeomorphism. The map $f$ is \textit{proper}, if the pre-image in $f$ of every compact set is compact. 

Let $f \colon X \ra Y$ be a proper branched covering between manifolds. Then the codimension of $B_f\subset X$ is at least two by V\"ais\"al\"a \cite{Vaisala} and the restriction map  
$$f':=f \r X \setminus f^{-1}(f(B_f)) \colon X \setminus f^{-1}(f(B_f)) \ra Y \setminus f(B_f)$$ is a covering map between open connected manifolds, see Church and Hemmingsen \cite{Church-Hemmingsen}. Thus there exists, by classical theory of covering maps, an open manifold $X_f'$ and a commutative diagram of proper branched covering maps 
\begin{equation*}
\xymatrix{
& X_f' \ar[ld]_{p'} \ar[rd]^{q'} &\\
X \ar[rr]^f & & Y }
\end{equation*} 
where $p' \colon X_f' \ra X \setminus f^{-1}(f(B_f))$ and $q' \colon X_f' \ra Y \setminus f(B_f)$ are normal covering maps and the deck-trans\-for\-ma\-ti\-on group of the covering map $q' \colon X_f' \ra Y \setminus f(B_f)$ is isomorphic to the mono\-dromy group of $f'.$ 

Further, by Berstein and Edmonds \cite{Berstein-Edmonds}, there exists a locally compact and locally 
connected second countable Hausdorff space $X_f \supset 
X_f'$ so that $X_f \setminus X_f'$ does not locally 
separate $X_f$ and the maps $p'$ and $q'$ extend to 
proper normal branched covering maps $p \colon X_f \ra X$ and $
\bar f:=q \colon X_f \ra Y$ so that the diagram
\begin{equation*}
\xymatrix{
& X_f \ar[ld]_{p} \ar[rd]^{\bar f} &\\
X \ar[rr]^f & & Y }
\end{equation*} 
commutes, and $p$ and $\bar f$ are the Fox-completions of $p'\colon X'_f \to X$ and $q' \colon X'_f \to Y,$ see also \cite{Aaltonen}, \cite{Fox} and \cite{Montesinos-old}. In this paper the triple $(X_f,p, \bar f)$ is called the \textit{monodromy representation,} $\bar f \colon X_f \to Y$ the \textit{normalization} and the space $X_f$ the \textit{monodromy space} of $f.$  

The monodromy space $X_f$ is a locally connected and locally compact Hausdorff space and, by construction, all points in the open and dense subset $X_f \setminus B_{\bar f} \subset X_f$ are manifold points. The natural question to ask regarding the monodromy space $X_f$ is thus the following: \textit{What does the monodromy space $X_f$ look like around the branch points of $\bar f$?} 

When $X$ and $Y$ are $2$-manifolds, Sto\"ilows Theorem implies, that the points in $B_{\bar f}$ are manifold points and the monodromy space $X_f$ is a manifold. We further know by Fox \cite{Fox} that the monodromy space $X_f$ is a locally finite simplicial complex, when $f \colon X \to Y$ is a simplicial branched covering between piecewise linear manifolds. It is, however, stated as a question in \cite{Fox} under which assumptions the locally finite simplicial complex obtained in Fox' completion process is a manifold. We construct here an example in which the locally finite simplicial complex obtained in this way is not a manifold. 

\begin{Thm}\label{eka}
There exists a simplicial branched cover $f \colon S^3 \to S^3$ for which the monodromy space $X_f$ is not a manifold.
\end{Thm}

Theorem \ref{eka} implies that the monodromy space is not in general a manifold even for proper simplicial branched covers between piecewise linear manifolds. Our second theorem states further, that in the non-piecewise linear case the monodromy space is not in general even a locally contractible space. We construct a branched covering, which is a piecewise linear branched covering in the complement of a point, but for which the monodromy space is not a locally contractible space.  

\begin{Thm}\label{toka}
There exists a branched cover $f \colon S^3 \to S^3$ for which the monodromy space $X_f$ is not a locally contractible space.  
\end{Thm}

We end this introduction with our results on the cohomological properties of the monodromy space. The monodromy space of a proper branched covering between manifolds is always a locally orientable space of finite cohomological dimension. However, in general the monodromy space is not a cohomology manifold in the sense of Borel \cite{Borel-book}; there exist a piecewise linear branched covering $S^3 \to S^3$ for which the monodromy space is not a cohomology manifold. This shows, in particular, that the theory of normalization maps of proper branched covers between manifolds is not covered by Smith-theory in \cite{Borel-book} and completes \cite{Aaltonen-Pankka} for this part.  

This paper is organized as follows. In Section \ref{tm}, we give an example $f \colon S^2 \to S^2$ of an open and discrete map for which the monodromy space is not a two sphere. In Section \ref{ts} we show that the suspension $\Sigma f \colon S^3 \to S^3$ of $f$ prove Theorem \ref{eka} and that the monodromyspace of $f$ is not a cohomology manifold. In Section \ref{pii} we construct an open and discrete map $g \colon S^3 \to S^3.$ In Section \ref{dp} we show that $g \colon S^3 \to S^3$ proves Theorem \ref{toka}. 

\subsection*{Aknowledgements} I thank my adviser Pekka Pankka for introducing me to the paper \cite{Heinonen-Rickman} by Heinonen and Rickman and for many valuable discussions on the topic of this paper.

\section{Preliminaries}

In this paper all topological spaces are locally connected and locally compact Hausdorff spaces if not stated otherwise. Further, all proper branched coverings $f \colon X \to Y$ between topological spaces are also branched coverings in the sense of Fox \cite{Fox} and completed coverings in the sense of \cite{Aaltonen}; $Y':=Y\setminus f(B_f)$ and $X':=X \setminus f^{-1}(B_f)$ are open dense subsets so that $X \setminus X'$ does not locally separate $X$ and $Y' \setminus Y$ does not locally separate $Y.$ We say that the proper branched covering $f \colon X \ra Y$ is \textit{normal}, if $f':=f\r X' \colon X' \ra Y'$ is a normal covering. By Edmonds \cite{Edmonds-Michigan} every proper normal branched covering $f \colon X \ra Y$ is an orbit map for the action of the deck-transformation group $\Deck(f)$ i.e. $X/\mathrm{Deck(f)} \approx Y.$ 

We recall some elementary properties of proper branched coverings needed in the forthcoming sections. Let $f \colon X \ra Y$ be a proper normal branched covering and $V \subset Y$ an open and connected set. Then each component of $f^{-1}(V)$ maps onto $V,$ see \cite{Aaltonen}. Further, if the pre-image $D:=f^{-1}(V)$ is connected, then $f \r D \colon D \ra V$ is a normal branched covering and the map
\begin{equation}\label{tf}
\Deck(f) \to \Deck(f \r D), \tau \mapsto \tau\r D,
\end{equation} 
is an isomorphism, see \cite{Montesinos-old}. 

\begin{Apulause}\label{1}Let $f \colon X \ra Y$ be a branched covering between manifolds. Suppose $p \colon W \ra X$ and $q \colon W \ra Y$ are normal branched coverings so that $q=p \circ f.$ Then $\Deck(p) \subset \Deck(q)$ is a normal subgroup if and only if $f$ is a normal branched covering.
\end{Apulause}

\begin{proof} Let $Y':=Y \setminus f(B_f),$ $X':=X \setminus f^{-1}(f(B_f))$ and $W'=W \setminus q^{-1}(f(B_f)).$ Let $f':=f \r X' \colon X' \to Y',$ $p':=p \colon W' \ra X'$ and $q':= q \r W'\colon W' \to Y'$ and let $w_0 \in W',$ $x_0=p'(w_0)$ and $y_0=q'(w_0).$ Then $\mathrm{Deck}(p) \subset \mathrm{Deck}(q)$ is a normal subgroup if and only if $\mathrm{Deck}(p') \subset \mathrm{Deck}(q')$ is a normal subgroup and the branched covering $f$ is a normal branched covering if and only if the covering $f'$ is normal. We have also a commutative diagram 
\begin{equation*}
\xymatrix{
& W' \ar[ld]_{p'} \ar[rd]^{q'} &\\
X' \ar[rr]^{f'} & & Y'&}
\end{equation*} 
%\begin{equation*}\label{rr}
%\xymatrix{
%& \overline{X}' \ar[ld]_{p'} \ar[rd]^{q'} &\\
%X' \ar[rr]^{f'} & & Z'}
%\end{equation*} 
of covering maps, where 
$$q_*'(\pi_1(W',w_0)) \subset f'_*(\pi_1(X', x_0)) \subset \pi_1(Y',y_0).$$  The deck-homomorphism $$\pi_{(q',_{w_0})} \colon \pi_1(Y',y_0) \to \mathrm{Deck}(q')$$ now factors as 
\begin{equation*}
\xymatrix{
\pi_1(Y',y_0) \ar[rr]^{\pi_{(q',w_0)}} \ar[dr]& & \mathrm{Deck}(q')\\
& \pi_1(Y',y_0)/q_*'(\pi_1(W',w_0)) \ar[ru]^{\ol{\pi}_{(q',w_0)}},&}
\end{equation*}  
for an isomorphism $\ol{\pi}_{(q',w_0)}$ and
$$\pi_{(q',_{w_0})}(f'_* (\pi_1(X',x_0)))=\ol{\pi}_{(q',w_0)} \big(f'_*(\pi_1(X',x_0))/q_*'(\pi_1(W', w_0))\big)=\mathrm{Deck}(p').$$ 
In particular, $\mathrm{Deck}(p') \subset \mathrm{Deck}(q')$ is a normal subgroup if and only if
$$f'_*(\pi_1(X',x_0))/q_*'(\pi_1(W', w_0)) \subset \pi_1(Y',y_0)/q_*'(\pi_1(W',w_0))$$ is a normal subgroup. Since $q_*'(\pi_1(W',w_0)) \subset f'_*(\pi_1(X', x_0)),$ this implies that 
$\mathrm{Deck}(p') \subset \mathrm{Deck}(q')$ is a normal subgroup if and only if $$f'_*(\pi_1(X',x_0)) \subset \pi_1(Y',y_0)$$ is a normal subgroup. We conclude that $f$ is a normal branched covering if and only if $\mathrm{Deck}(p) \subset \mathrm{Deck}(q)$ is a normal subgroup.  
\end{proof}

Let $f \colon X \ra Y$ be a proper branched covering. We say that $D \subset X$ is a \textit{normal neighbourhood} of $x$ if $f^{-1}\{f(x)\} \cap D=\{x\}$ and $f \r D \colon D \ra V$ is a proper branched covering. We note that for every $x \in X$ there exists a neighbourhood $U$ of $f(x)$ so that for every open connected neighbourhood $V \subset U$ of $f(x)$ we have the $x$-component $D$ of $f^{-1}(V)$ a normal neighbourhood of $x$ and the pre-image $f^{-1}(E) \subset X$ is connected for every open connected subset $E \subset Y$ satisfying $Y \setminus E \subset U.$ This follows from the following lemma.

\begin{Apulause}\label{fun} Let $f \colon X \to Y$ be a proper branched covering. Then for every $y \in Y$ there exists such a neighbourhood $U$ of $y$ that the pre-image $f^{-1}(V) \subset X$ is connected for every open connected set $V \subset Y$ satisfying $Y \setminus V \subset U.$
\end{Apulause}

\begin{proof} Let $y_0 \in Y \setminus \{y\}.$ Since $f$ is proper the subsets $f^{-1}\{y\}, f^{-1}\{y_0\} \subset X$ are finite. Since $W \setminus f^{-1}\{y\}$ is connected, there exists a path $\gamma \colon [0,1] \to W \setminus f^{-1}\{y\}$ so that $f^{-1}\{y_0\} \subset \gamma[0,1].$ Let $U \subset Y$ be a neighbourhood of $y$ satisfying $U \cap f(\gamma[0,1])=\es.$ 

Suppose that $V \subset Y$ is an open connected subset satisfying $Y \setminus V \subset U.$
Then $f^{-1}\{y_0\} \subset \gamma[0,1]$ is contained in a component of $f^{-1}(V),$ since $f(\gamma[0,1]) \subset V.$ Since $V \subset Y$ is connected, every component of $f^{-1}(V)$ maps onto $V.$ Thus $f^{-1}(V)=D.$  
\end{proof}

We end this section with introduction the terminology and elementary results for the part of singular homology. Let $X$ be a locally compact and locally connected second countable Hausdorff space. In this paper $H_i(X; \Z)$ is the $i$:th singular homology group of $X$ and $\widetilde{H}_i(X;\Z)$ the $i$:th reduced singular homology group of $X$ with coefficients in $\Z,$ see \cite{Massey-book}. We recall that $H_i(X; \Z)=\widetilde{H}_i(X;\Z)$ for all $i\neq 0$ and $\widetilde{H}_0(X;\Z)=\Z^{k-1},$ where $k$ is the number of components in $X.$ We recall that for open subsets $U,V \subset X$ with $X=U \cup V$ and $X=U \cap V$ connected the reduced Mayer-Vietoris sequence is a long exact sequence of homomorphisms that terminates as follows:
$$\to H_1(X;\Z) \to \widetilde{H}_0(U \cap V;\Z) \to \widetilde{H}_0(U; \Z) \oplus \widetilde{H}_0(V; \Z) \to \widetilde{H}_0(X; \Z).$$

\section{Local orientability and cohomological dimension}
In this section we show that the monodromy space of a proper branched covering between manifolds is a locally orientable space of finite cohomological dimension. We also introduce Alexander-Spanier cohomology following the terminology of Borel \cite{Borel-book} and Massey \cite{Massey-book} and define a cohomology manifold in the sense of Borel \cite{Borel-book}.

Let $X$ be a locally compact and locally connected second countable Hausdorff space. In this paper $H_c^i(X; \Z)$ is the $i$:th Alexander-Spanier cohomology group of $X$ with coefficients in $\Z$ and compact supports. Let $A \subset X$ be a closed subset and $U=X \setminus A$. The standard push-forward homomorphism $H^i_c(U;\Z) \to H^i_c(X;\Z)$ is denoted $\tau^i_{XU},$ the standard restriction homomorphism $H^i_c(X;\Z)\to H^i_c(U;\Z)$ is denoted $\iota^i_{UX}$ and the standard boundary homomorphism $H^i_c(A; \Z) \to H^{i+1}_c(X \setminus A; \Z)$ is denoted $\partial_{(X/A)A}^i$ for all $i \in \N.$ 

We recall that the exact sequence of the pair $(X,A)$ is a long exact sequence
$$\to H^i_c(X \setminus A; \Z) \to H^i_c(X; \Z) \to H^i_c(A; \Z) \to H^{i+1}_c(X \setminus A; \Z)\to $$
where all the homomorphisms are the standard ones. We also recall that $\tau^i_{VX}=\tau^i_{XU} \circ \tau^i_{UV}$ for all open subsets $V \subset U$ and $i \in \N.$ 

%Further we recall the Mayer-Vietoris sequence. Let $U,V \subset X$ be open subsets so that $X=U \cup V,$ then the Mayer Vietoris sequence is a long exact sequence
%$$ \cdots \to H^{i-1}_c(X; \Z) \to^{\Delta_1^{i-1}} H^i_c(U \cap V; \Z) \to^{\varphi_1^i} H^i_c(U; \Z) \oplus H^i_c(V; \Z) \to \cdots$$
%of homomorphisms, where 
%$$\varphi_1^i(c)=(\tau^i_{V(U \cap V)}(c),\tau^i_{U(U \cap V)}(c))$$ for all $c \in H^i_c(U \cap V; \Z)$ and $i \in \N.$ The Mayer–Vietoris sequence is natural in the sense that, if $X' \subset X$ and $U' \subset U$ and $V' \subset V$ are open subsets so that $X'=U' \cup V'$, then the squares in the diagram
%\begin{equation*}\label{hei}
%\xymatrix{\ar[r]& \ar[d]_{\tau_{XX'}} \ar[r] H^{i-1}_c(X'; \Z)& \ar[d]_{\tau_{(U \cap V)(U' \cap V')}} H^i_c(U' \cap V'; \Z) \ar[r]& \ar[d]_{(\tau_{UU'},\tau_{VV'})} \ar[r] H^i_c(U'; \Z) \oplus H^i_c(V'; \Z) &\\
%\ar[r]& H^{i-1}_c(X; \Z)\ar[r] & H^i_c(U \cap V; \Z) \ar[r] &  \ar[r] H^i_c(U; \Z) \oplus H^i_c(V; \Z)&\\
%}
%\end{equation*} 
%commute. 

The \textit{cohomological dimension} of a locally compact and locally connected Hausdorff space $X$ is $\leq n,$ if $H_c^{n+1}(U ;\Z)=0$ for all open subsets $U \subset X.$   

\begin{Thm} Let $f \colon X \to Y$ be a proper branched covering between $n$-manifolds. Then the monodromy space $X_f$ of $f$ has dimension $\leq n.$ 
\end{Thm}

\begin{proof} Let $B_{\bar f} \subset X_f$ be the branch set of the normalization map $\ol{f} \colon X_f \to Y$ of $f.$ Let $U \subset X_f$ be a connected open subset and $B_{\bar f \r U}$ the branch set of $\bar f \r U.$  The cohomological dimension of $B_{\bar f \r U}$ is at most $n-2$ by \cite{Borel-book}, since $B_{\bar f \r U}$ does not locally separate $U.$ Thus $H^{i}_c(B_{\bar f \r U}; \Z)=0$ for $i > n-2$ and the part
$$ \to H^{i-1}_c(B_{\bar f \r U}; \Z) \to H^i_c(U \setminus B_{\bar f \r U}; \Z) \to H^i_c(U; \Z) \to H^i_c(B_{\bar f \r U}; \Z)\to $$
of the long exact sequence of the pair $(U,B_{\bar f \r U})$ gives us an isomorphism $H^i_c(U \setminus B_{\bar f \r U}; \Z) \to H^i_c(U; \Z)$ for $i \geq n.$ Since $U \setminus B_{\bar f \r U}$ is a connected $n$-manifold, $H^{n+1}_c(U \setminus B_{\bar f \r U}; \Z)=0.$ Thus 
$H^{n+1}_c(U; \Z) \cong H^{n+1}_c(U \setminus B_{\bar f \r U}; \Z) =0.$ We conclude that $X_f$ has dimension $\leq n.$
%cohomological dimension of $X_f \setminus B_{\bar f}$ is $n$ as an $n$-manifold, $H^i_c(X_f; \Z)\cong H^i_c(X \setminus B_{\bar f}; \Z)=0$ for $i > n$ and $H^n_c(X_f; \Z)\cong H^n_c(X \setminus B_{\bar f}; \Z) \neq 0.$  
\end{proof}

The $i$:th \textit{local Betti-number} $\rho^i(x)$ around $x$ is $k,$ if given a neighbourhood $U$ of $x,$ there exists open neighbourhoods $W \subset V \subset U$ with $\bar W \subset V$ and $\bar V \subset U$ so that $\mathrm{Im}(\tau_{VW'}^i)=\mathrm{Im}(\tau_{VW'}^i)$ and has rank $k.$ The space $X$ is called a \textit{Wilder manifold}, if $X$ is finite dimensional and for all $x \in X$ the local Betti-numbers satisfy $\rho^i(x)=0$ for all $i < n$ and $\rho^n(x)=1.$ 

A locally compact and locally connected Hausdorff space $X$ with cohomological dimension $\leq n$ is \textit{orientable}, if there exists for every $x \in X$ a neighbourhood basis $\mathcal{U}$ of $x$ so that $\mathrm{Im}(\tau^n_{XU})=\Z$ for all $U \in \mathcal{U},$ and locally orientable if every point in $X$ has an orientable neighbourhood.
  
\begin{Thm}\label{harmautta} Let $\bar f \colon X_f \to Y$ be a normalization map of a branched covering $f \colon X \to Y$ so that $Y$ is orientable. Then $X_f$ is orientable.
\end{Thm}

\begin{proof} The set $Y \setminus \bar f (B_{\bar f})$ is an open connected subset of the orientable manifold $Y.$ Thus $X_f \setminus B_{\bar f}$ is an orientable manifold as a cover of the orientable manifold $Y \setminus \bar f (B_{\bar f}).$ Thus $H_c^n(X_f';Z)=\Z.$ 

We show that $H_c^n(X_f;\Z)=\Z$ and that the push-forward $\tau_{X_fW}$ is an isomorphism for every $x \in X_f$ and normal neighbourhood $W$ of $x.$ The cohomological dimension of $B_{\bar f}$ is $\leq 2.$ Thus, by the long exact sequences of the pairs $(X_f,B_{\bar f})$ and $(W,B_{\bar f} \cap W),$ the push-forward homomorphisms $\tau^n_{X_f(X_f \setminus B_{\bar f})}$ and $\tau^n_{W(W \setminus B_{\bar f})}$ are isomorphisms. Since $W \setminus B_ {\bar f} \subset X_f \setminus B_{\bar f}$ is a connected open subset and $X_f \setminus B_{\bar f}$ is a connected orientable $n$-manifold $,$ the push-forward homomorphism $\tau_{(X_f \setminus B_{\bar f})((W \setminus B_{\bar f}))}$ is an isomorphism.  
Since $\tau^n_{X_fW} \circ \tau^n_{W(W \setminus B_{\bar f})}=\tau^n_{X_f(X_f \setminus B_{\bar f})} \circ \tau^n_{(X_f \setminus B_ {\bar f})(W \setminus B_{\bar f})},$ we conclude that $\tau^n_{X_fW}$ is an isomorphism and $\mathrm{Im}(\tau^n_{X_fW})=\mathrm{Im}(\tau^n_{X_fX_f'})\cong \Z.$
\end{proof}

We note that a similar argument shows that a monodromy space of a proper barnched covering between manifolds is always locally orientable. A \textit{cohomology manifold} in the sense of Borel \cite{Borel-book} is a locally orientable Wilder manifold.

\section{The monodromy space of branched covers between surfaces}\label{tm}
A \textit{surface} is a closed orientable $2$-manifold. The monodromy space related to a branched covering between surfaces is always a surface as mentioned in \cite{Fox}. We first present the proof of this fact in the case we use it for completion of presentation and then we show that there exists a branched cover $f \colon S^2 \to S^2$ so that $X_f\neq S^2$ towards proving Theorems \ref{eka} and \ref{kolmas}. %In this section we give for a condition for $X_f \neq S^2$ for branched covers onto the two sphere $S^2.$ 
%$F \to S^2$ show that there is a branched cover $S^2 \to S^2$ for which the monodromy space is not $S^2.$
%The classical Sto\"ilow's theorem states the following: \emph{Let $f\colon F \to S$ 
%be a branched cover $f\colon F \to S$ between surfaces. Then the branched set $B_f \subset F$ is discrete and, for each point $z\in B_f$, there exists a neighbourhood $U$, homeomorphisms 
%$\phi \colon U \to \C$ and $\psi \colon f(U) \to \C$, and $k\in \N$ for which $\psi \circ f \circ 
%\psi^{-1}$ is the power map $z\mapsto z^k.$}; see e.g; \cite{Whyburn-book} for a proof. 
%In particular, as an almost direct consequence of Sto\"ilow's theorem we get the following lemma. 

\begin{Apulause}\label{r} Let $F$ be an orientable surface and $f \colon F \to S^2$ be a proper branched cover and $\overline{f} \colon X_f \to S^2$ the normalization of $f$. Then $X_f$ is an orientable surface.
\end{Apulause}

\begin{proof} Since $S^2$ is orientable, the space $X_f$ is orientable by Theorem \ref{harmautta}. Since the domain $F$ of $f$ is compact, the normalization $\overline{f}$ has finite multiplicity and the space $X_f$ is compact. Let $x \in X_f.$ By  Sto\"ilow's theorem, see \cite{Whyburn-book}, $f(B_f)=\ol{f}(B_{\ol{f}})$ is a discrete set of points. Thus there exists a normal neighbourhood $V \subset X_f$ of $x$ so that $\ol{f}(V) \cap f(B_{f}) \subset \{f(x)\}$ and $\overline{f}(V )\approx \R^2.$ Now $\overline{f}\r V \setminus \{x\} \colon V \setminus \{x\}  \to \overline{f}(V \setminus \{f(x)\})$ is a cyclic covering of finite multiplicity, since $\overline{f}(V \setminus \{f(x)\})$ is homeomorphic to the complement of a point in $\R^2.$ We conclude from this that $x$ is a manifold point of $X_f.$ Thus $X_f$ is a $2$-manifold and a surface.
\end{proof}

We record as a theorem the following result in the spirit of Fox \cite[p.255]{Fox}.

\begin{Thm}\label{s} Let $F$ be an orientable surface and $f \colon F \to S^2$ a proper branched cover and $\bar f \colon X_f \to S^2$ the normalization of $f$. Assume $|fB_f| > 3.$ Then $X_f \neq S^2.$
\end{Thm}

\begin{proof} The space $X_f$ is $S^2$ if and only if the Euler characteristics $\chi(X_f)$ is $2$. By Riemann Hurwitz formula  
\begin{equation*}
\label{d}
\chi(X_f) = (\deg \bar f)\chi(S^2)-\sum_{x \in X_f}(i(x,\bar f)-1), 
\end{equation*}
where $i(x,\bar f)$ is the local index of $\bar f$ at $x$. Since $\ol{f}$ is a normal branched cover,  $i(x',\bar f)=i(x,\bar f)$ for $x,x' \in X_f$ with $\bar f (x)= \bar f (x').$ We define for all $y \in S^2,$
$$n(y):=i(x,\bar f),x \in \ol{f}^{-1}\{y\}.$$ 
Then for all $y \in S^2$ 
$$\deg \bar f = \sum_{x\in \bar f^{-1}\{y\}} i(x,\bar f) =n(y)|\bar f ^{-1}\{y\}|$$ and thus for all $y \in S^2$  
$$|\bar f^{-1}\{y\}|=\dfrac{\deg \bar f}{n(y)}.$$ Hence,
\begin{equation*}
\label{f}
\chi(X_f) = (\deg \bar f)\left(\chi(S^2)-\sum_{y \in \bar fB_{ \bar f}}\dfrac{n(y)-1}{n(y)}\right), 
\end{equation*}
where $\chi(S^2)=2$ and $(\deg \bar f):=N \in \N.$ Since $n(y)\geq 2$ for all $y \in \bar f B_{\bar f}=fB_f$ and $\dfrac{k-1}{k} \to 1$ as $k \to \infty,$ we get the estimate
\begin{equation*}
\chi(X_f) \leq N\left(2-|fB_f|\dfrac{1}{2}\right). 
\end{equation*}
Thus $\chi(X_f)\leq 0 < 2,$ since $|fB_f|\geq 4$ by assumption. Thus $X_f \neq S^2.$ 
\end{proof}

We end this section with two independent easy corollaries. 

\begin{Cor} Let $F$ be an orientable surface and $f \colon F \to S^2$ be a proper branched cover so that $|fB_f| > 3.$ Then $f$ is not a normal covering.
\end{Cor}

\begin{Cor} Let $F$ be an orientable surface and $f \colon F \to S^2$ be a proper branched cover so that $|fB_f| < 3.$ Then $f$ is a normal covering.
\end{Cor}
\begin{proof} Since the first fundamental group of $S^2\setminus fB_f$ is cyclic, the monodromy group of $f \r F \setminus f^{-1}(fB_f) \colon F \setminus f^{-1}(fB_f) \to S^2\setminus fB_f$ is cyclic. Thus every subgroup of the deck-transformation group of the normalization map $\bar f \colon X_f \to X$ is a normal subgroup. Thus $f=\bar f$ and in particular, $f$ is a normal branched covering. 
\end{proof}

\section{The suspension of a branched cover between orientable surfaces}\label{ts}

In this section we prove Theorem \ref{eka} in the introduction and the following theorem. 

\begin{Thm}\label{kolmas}
There exists a simplicial branched cover $f \colon S^3 \to S^3$ for which the monodromy space $X_f$ is not a cohomology manifold.
\end{Thm}

More precisely, we show that there exists a branched cover $S^2 \to S^2$ for which the monodromy space is not a manifold or a cohomology manifold for the suspension map $\Sigma S^2 \to \Sigma S^2.$ 

Let $F$ be an orientable surface. Let $\sim$ be the equivalence relation in $F \times [-1,1]$ defined by the relation $(x,t) \sim (x',t)$ for $x,x' \in F$ and $t \in \{-1,1\}.$ Then the quotient space $\Sigma F:=F \times [-1,1]/\sim$ is the \textit{suspension space} of $F$ and the subset $CF:=\{\overline{(x,t)}: x \in F, t \in [0,1]\} \subset \Sigma F$ is the \textit{cone} of $F.$ We note that $\Sigma S^2 \approx S^3.$ Let $f \colon F_1 \to F_2$ be a piecewise linear branched cover between surfaces. Then $\Sigma f \colon \Sigma F_1 \to \Sigma F_2, \overline{(x,t)} \mapsto \overline{(f(x),t)},$ is a piecewise linear branched cover and called the \textit{suspension map} of $f.$ We note that the suspension space $\Sigma F$ is a polyhedron and locally contractible for all surfaces $F.$

We begin this section with a lemma showing that the normalization of a suspension map of a branched cover between surfaces is completely determined by the normalization of the original map. 

\begin{Apulause}\label{t} Let $F$ be an orientable surface and $f \colon F \to S^2$ a branched cover and $\bar f \colon X_f \to S^2$ the normalization of $f.$ Then $\Sigma \bar f \colon \Sigma X_f \to \Sigma S^2$ is the normalization of $\Sigma f \colon \Sigma F \to \Sigma S^2.$
\end{Apulause}

\begin{proof}Let $p \colon X_f \to F$ be a normal branched covering so that $\bar f=f \circ p.$  Then $\Sigma \bar f$ is a normal branched cover so that $\Sigma \bar f = \Sigma p \circ \Sigma f$ and $\varphi \colon \mathrm{Deck}(\Sigma \bar f) \to \mathrm{Deck}(\bar f), \tau \mapsto \tau \r X_f$ is an isomorphism. We need to show that, if $G \subset \mathrm{Deck}(\Sigma \bar f)$ is a subgroup so that $f \circ (\Sigma p/G)$ is normal. Then $G$ is trivial.

Suppose $G \subset \mathrm{Deck}(\Sigma \bar f)$ is a group so that $f \circ \Sigma/G$ is normal. Then $f \circ (p/G')$ is normal for the quotient map $p/G' \colon X_f/G' \to S^2,$ where $G'=\varphi(G).$ Since $\bar f$ is the normalization of $f,$ the group $G'$ is trivial. Thus $G=\varphi^{-1}(G')$ is trivial, since $\varphi^{-1}$ is an isomorphism.
\end{proof}

We then characterize the surfaces for which the suspension space is a manifold or a cohomology manifold in the sense of Borel.

\begin{Apulause}\label{y}Let $F$ be an orientable surface. Then $\Sigma F$ is a manifold if and only if $F=S^2.$
\end{Apulause}
\begin{proof} Suppose $F=S^2.$ Then $\Sigma F \approx S^3.$ Suppose then that $F\neq S^2.$ Then there exists a (cone) point $x \in \Sigma F$ and a contractible neighbourhood $V \subset \Sigma F$ of $x$ so that $F \subset V$ and $V \setminus \{x\}$ contracts to $F.$ Now $\pi_1(V \setminus \{x\},x_0) \cong \pi_1(F,x_0)\neq 0$ for $x_0 \in F.$ Suppose that $\Sigma F$ is a $3$-manifold. Then $\pi_1(V \setminus \{x\},x_0)=\pi_1(V,x_0)=0,$ which is a contradiction. Thus $\Sigma F$ is not a manifold. 
\end{proof}

\begin{Apulause}\label{p} Let $F$ be an orientable surface. Then $\Sigma F$ is not a Wilder manifold if $F\neq S^2.$
\end{Apulause}

\begin{proof} We show that then the second local Betti-number is non-trivial around a point in $\Sigma F.$ Let $CF \subset \Sigma F$ be the cone of $F.$ Let $\pi \colon  F \times [0,1] \to \Sigma F, (x,t) \mapsto \overline{(x,t)},$ be the quotient map to the suspension space and $\bar x= \pi (F \times \{1\}).$ We first note that $H_c^1(CF,\Z)=H_c^2(CF,\Z)=0,$ since $CF$ contracts properly to a point. Further, by Poincare duality $H_c^1(F; \Z)=\Z^{2g},$ where $g$ is the genus of $F.$ In the exact sequence of the pair $(CF,F)$ we have the short exact sequence
$$\to 0 \to H^2_c(CF \setminus F; \Z) \to H_c^1(F; \Z) \to 0.$$
Thus $H^2_c(CF \setminus F; \Z) \cong H_c^1(F; \Z)$ and $H^2_c(CF \setminus F; \Z)=0$ if and only if $g=0$ for the genus $g$ of $F.$ Thus $H^2_c(CF \setminus F; \Z)=0$ if and only if $F=S^2.$

We then show that the rank of $H^2_c(CF \setminus F; \Z)$ is the local Betti-number $\rho^2(\bar x)$ around $\bar x.$ For this it is sufficient to show that given any neighbourhood $U\subset CF$ of $\bar x,$ there exists open neighbourhoods $W \subset V \subset U$ of $\bar x$ with $\bar W \subset V, \bar V \subset U,$ so that for any open neighbourhood $W' \subset W$ of $\bar x,$ $\mathrm{Im}(\tau_{VW})=\mathrm{Im}(\tau_{VW'})\cong H^2_c(CF \setminus F; \Z).$ 

Denote $\Omega_{t}=\varphi(F \times [0,t))$ for all $t \in (0,1).$ We note that then $\tau_{\Omega_s \Omega_t}$ is an isomorphism for all $t,s \in \R, t<s,$ since $\iota \colon \Omega_t \to \Omega_s$ is properly homotopic to the identity. Let $U \subset CF$ be any neighbourhood of $\bar x.$ We set $V=\Omega_t$ for such $t \in (0,1)$ that $\Omega_t \subset U$ and $W=\Omega_{t/2}.$ Then for any neighbourhood $W' \subset \Omega_t$ of $\bar x$ there exists $t' \in (0,t/2)$ so that $\Omega_{t'} \subset W'.$ Now $\tau_{\Omega_tW'}$ is surjective, since $\tau_{\Omega_t\Omega_{t'}}=\tau_{\Omega_t W'} \circ \tau_{W'\Omega_ {t'}}$ is an isomorphisms. Thus $$\mathrm{Im}(\tau_{\Omega_t\Omega_{t/2}})=\mathrm{Im}(\tau_{\Omega_t\Omega_{t'}})= H^2_c(\Omega_t; \Z) \cong H^2_c(CF \setminus F;\Z),$$
since $\tau_{\Omega_t\Omega_{t/2}}$ and $\tau_{\Omega_t(CF/F)}$ are isomorphisms. Thus $\rho^2(\bar x)\neq 0,$ since $F \neq S^2.$ 
\end{proof}

\begin{Cor}\label{89} Let $f \colon S^2 \to S^2$ be a branched cover with $|fB_f|> 3,$ $\Sigma f \colon S^3 \to S^3$ the suspension map of $f$ and $\overline{\Sigma f} \colon X_{\Sigma f} \to S^3$ the normalization of $\Sigma f.$ Then $X_{\Sigma f}$ is not a manifold and not  a Wilder manifold.   
\end{Cor}

\begin{proof} By Lemmas \ref{r} and \ref{t} we know that $X_f$ is a surface and $X_{\Sigma f}=\Sigma X_f.$ Further, by Lemma \ref{s} we know that $X_f \neq S^2,$ since $|fB_f|> 3.$ Thus, by Lemma \ref{y}, $X_f$ is not a manifold. Further, by Lemma \ref{p}, $X_{\Sigma f}$ is not a Wilder manifold. Thus $X_f$ is not a cohomology manifold in the sense of Borel.
\end{proof}

\begin{proof}[Proof of Theorems \ref{eka} and \ref{kolmas}] By Corollary \ref{89} it is sufficient to show that there exists a branched cover $f \colon S^2 \to S^2$ so that $|f(B_{f})|=4.$ Such a branched cover we may easily construct as $f=f_1 \circ f_2,$ where $f_i \colon S^2 \to S^2$ is a winding map with branch points $x_1^i$ and $x_1^i$ for $i \in \{1,2\}$ satisfying $x_1^2, x_2^2 \notin \{f_1(x_1^1),f_1(x_2^1)\}$ and $f_2(f_1(x_1^1))\neq f_2(f(x_1^1)).$  
\end{proof}
 
\section{An example of a non-locally contractible monodromy space}\label{pii}

In this section we introduce an example of a branched cover $S^3 \to S^3$ for which the related monodromy space is not a locally contractible space. The construction of the example is inspired by Heinones and Rickmans construction in \cite{Heinonen-Rickman} of a branched covering $S^3 \to S^3$ containing a wild Cantor set in the branch set. 
We need the following result originally due to Berstein and Edmonds \cite{ONT} in the extent we use it.

\begin{Thm}[\cite{PKJ}, Theorem 3.1]\label{PKJ}Let $W$ be a connected, compact, oriented piecewise linear 3-manifold whose boundary consists of $p \geq 2$ components $M_0, \ldots, M_{p-1}$ with the induced orientation. Let $W'=N \setminus int B_j$ be an oriented piecewise linear $3$-sphere $N$ in $\R^4$ with $p$ disjoint, closed, polyhedral $3$-balls removed, and have the induced orientation on the boundary. Suppose that $n \geq 3$ and $\varphi_j \colon M_j \ra \partial B_j$ is a sense preserving piecewise linear branched cover of degree $n,$ for each $j=0,1, \ldots, p-1.$ Then there exists a sense preserving piecewise linear branched cover $\varphi \colon W \ra W'$ of degree $n$ that extends $\varphi_j:$s.
\end{Thm}

\begin{figure}[htb]
\includegraphics{RigidityPoster1.5}
\caption{}\label{rer}
\end{figure} 

Let $x \in S^3$ be a point in the domain and $y \in S^3$ a point in the target. Let $X \subset S^3$ be a closed piecewise linear ball with center $x$ and let $Y \subset S^3$ be a closed piecewise linear ball with center $y.$ Let $T_0 \subset \mathrm{int} X$ be a solid piecewise linear torus so that $x \in \mathrm{int} T_0.$ Now let $\mathcal{T}=(T_{n})_{n \in \N}$ be a sequence of solid piecewise linear tori in $T_0$ so that $T_{k+1} \subset \mathrm{int} T_{k}$ for all $k \in \N$ and $\Cap_{n=1}^{\infty}T_{n}=\{x\}.$ Let further $B_0 \subset \mathrm{int} Y$ be a closed piecewise linear ball with center $y$ and let $\mathcal{B}=(B_{n})_{n \in \N}$ be a sequence of closed piecewise linear balls with center $y$ so that $B_{k+1} \subset \mathrm{int} B_{k}$ for all $k \in \N$ and $\Cap_{n=1}^{\infty}B_{n}=\{y\}.$ See illustration in Figure \ref{rer}.

We denote $\partial X=\partial T_{-1}$ and $\partial Y=\partial B_{-1}$ and choose an orientation to all boundary surfaces from an outward normal. Let $f_n \colon \partial T_n \ra \partial B_n, n \in \{-1\} \cup \N,$ be a collection of sense preserving piecewise linear branched coverings so that 
\begin{itemize}
\item[(i)]the degree of all the maps in the collection are the same and greater than $2,$
\item[(ii)]$f_{-1}$ has an extension to a branched covering $g \colon S^3 \setminus \mathrm{int}X \to S^3 \setminus \mathrm{int}Y,$ 
\item[(iii)]the maps $f_n$ are for all even $n \in \N$ normal branched covers with no points of local degree three and
\item[(iv)]the branched covers $f_n$ have for all uneven $n \in \N$ a point of local degree three.
\end{itemize}
We note that for an example of such a collection of maps of degree $18,$ we may let $f_{-1}$ be a $18$-to-$1$ winding map, $f_n$ be for even $n \in \N$ as illustrated in Figure \ref{wer} and $f_n$ be for all uneven $n \in \N$ as illustrated in Figure \ref{ter}. 

\begin{figure}[htb]
\includegraphics{RigidityPoster1.3}
\caption{}\label{wer}
\end{figure} 

\begin{figure}[hbt]
\includegraphics{RigidityPoster1.2}
\caption{}\label{ter}
\end{figure}

Let then $n \in \N$ and let $F_n \subset X$ be the compact piecewise linear manifold with boundary $\partial T_{n-1} \cup \partial T_{n}$ that is the closure of a component of $X \setminus (\Cup_{k=-1}^\infty \partial T_k).$ Let further, $G_n \subset Y$ be the compact piecewise linear manifold with boundary $\partial B_{n-1} \cup \partial B_{n}$ that is the closure of a component of $Y \setminus (\Cup_{k=-1}^\infty \partial B_k).$ Then $F_n \subset X$ is a compact piecewise linear manifold with two boundary components and $G_n \subset Y$ is the complement of the interior of two distinct piecewise linear balls in $S^3.$ Further, $f_{n-1} \colon \partial T_{n-1} \ra \partial B_{n-1}$ and $f_{n} \colon \partial T_{n} \ra \partial B_{n}$ are sense preserving piecewise linear branched covers between the boundary components of $F_n$ and $G_n.$ Since the degree of $f_n$ is the same as the degree of $f_{n-1}$ and the degree is greater than $2,$ there exists by \ref{PKJ} a piecewise linear branched cover $g_n \colon F_n \to G_n$ so that $g_n \r \partial T_{n-1}=f_{n-1}$ and $g_n \r \partial T_{n}=f_{n}.$ 

Now $X=\Cup_{k=0}^\infty F_n$ and $Y=\Cup_{k=0}^\infty G_n$ and $g \colon S^3 \setminus \mathrm{int}X \to S^3 \setminus \mathrm{int}Y$ satisfies $g \r \partial X=g_{0} \r \partial X.$ Hence we may define a branched covering $f \colon S^3 \ra S^3$ by setting $f(x)=g_n(x)$ for $x \in G_n, n \in \N,$ and $f(x)=g(x)$ otherwise. 

However, we want the map $f \colon S^3 \to S^3$ to satisfy one more technical condition, namely the existence of collections of properly disjoint open sets $(M_k)_{k \in \N}$ of $X$ and $(N_k)_{k \in \N}$ of $Y$ so that $M_k \subset X$ is a piecewise linear regular neighbourhood of $\partial T_k$ and $N_k \subset Y$ is a piecewise linear regular neighbourhood of $B_k$ and $M_k=f^{-1}N_k,$ and $f \r M_k \colon M_k \to N_k$ has a product structure of $f_k$ and the identity map for all $k \in \N.$ We may require this to hold for the $f \colon S^3 \to S^3$ defined, since in other case we may by cutting $S^3$ along the boundary surfaces of $\partial T_k$ and $\partial B_k$ and adding regular neighbourhoods $M_k$ of $\partial T_k$ and $N_k$ of $ \partial B_k$ in between for all $k \in \N$ arrange this to hold without loss of conditions (i)--(iv), see \cite{RourkeSanderson}. 

In the last section of this paper we prove the following theorem. 

\begin{Thm}\label{11} Let $f \colon S^3 \ra S^3$ and $y \in S^3$ be as above and $\ol{f} \colon X_f \to Y$ the normalization of $f.$ Then $H_1(W; \Z) \neq 0$ for all open sets $W \subset X_f$ satisfying $$\overline{f}^{-1}\{y\} \cap W \neq \es.$$ 
\end{Thm}

Theorem \ref{eka} in the introduction then follows from Theorem \ref{11} by the following easy corollary.

\begin{Cor}\label{pimia} Let $f \colon S^3 \ra S^3$ and $y \in S^3$ be as above. Then the monodromy space $X_f$ of $f$ is not locally contractible.  
\end{Cor}

\begin{proof} Let $x \in \bar f ^{-1} \{y\}$ and $W$ a neighbourhood of $x.$ Then $H_1(W; \Z)\neq 0$ and $W$ has non-trivial fundamental group by Hurewich Theorem, see \cite{Massey-book}. Thus $W$ is not contractible. Thus the monodromy space $X_f$ of $f$ is not a locally contractible space. 
\end{proof}

\section{Destructive points}\label{dp}

In this section we define destructive points and prove Theorem \ref{11}. 

Let $X$ be a locally connected Hausdorff space. We call an open and connected subset $V \subset X$ a \textit{domain}. Let $V \subset X$ be a domain. A pair $\{A,B\}$ is called a \textit{domain covering} of $V,$ if $A,B \subset X$ are domains and $V=A \cup B.$ We say that a domain covering $\{A,B\}$ of $V$ is \textit{strong}, if $A \cap B$ is connected. Let $x \in V$ and let $U \subset V$ be a neighbourhood of $x.$ Then we say that $\{A,B\}$ is $U$-\textit{small} at $x$, if $x \in A \subset U$ or $x \in B \subset U.$

Let then $f \colon X \to Y$ be a branched covering between manifolds, $y \in Y$ and $V_0 \subset Y$ a domain containing $y.$ Then $V_0$ is a \textit{destructive neighbourhood} of $y$ with respect to $f,$ if $f \r f^{-1}(V_0)$ is not a normal covering to its image, but there exists for every neighbourhood $U \subset V_0$ of $y$ a $U$-small strong domain covering $\{A,B\}$ of $V_0$ at $y$ so that 
$\{f^{-1}(A),f^{-1}(B)\}$ is a strong domain cover of $f^{-1}(V_0)$ and $f \r (f^{-1}(A) \cap f^{-1}(B))$ is a normal covering to its image.

We say that $y$ is a \textit{destructive} point of $f$, if $y$ has a neighbourhood basis consisting of neighbourhoods that are destructive with respect to $f.$ 

\begin{Thm}\label{45}The map $S^3 \to S^3$ of the example in section \ref{pii} has a destructive point.
\end{Thm}
\begin{proof} We show that $y \in \Cap_{n=1}^{\infty}B_{n}$ is a destructive point of $f$. We first show that $V_0=\mathrm{int}B_0$ is a destructive neighbourhood of $y.$ 

We begin this by showing that $g:=f \r f^{-1}(V_0) \colon f^{-1}(V_0) \to V_0$ is not a normal branched cover. Towards contradiction suppose that $g$ is a normal branched cover. Then $\mathrm{Deck}(g) \cong \mathrm{Deck}(g \r M_1)$ and $\mathrm{Deck}(g) \cong \mathrm{Deck}(g \r M_2),$ since $M_1=f^{-1}(N_1)$ and $M_2=f^{-1}(N_2)$ are connected. On the other hand (iii) and (iv) imply that $\mathrm{Deck}(g \r M_1) \ncong \mathrm{Deck}(g \r M_2)$ and we have a contradiction. 

Let then $V_1 \subset V_0$ be any open connected neighbourhood of $y.$ Then there exists such $k \in \N,$ that $B_{2k} \cup N_{2k} \subset V_1.$ Let $B:=B_{2k} \cup N_{2k}$ and $A=(V_0 \setminus B_{2k}) \cup N_{2k}.$ Then $\{A,B\}$ is a strong $V_1$-small domain cover of $V_0$ at $y$ and $A \cap B=N_{2k}.$ In particular, $\{f^{-1}(A),f^{-1}(B)\}$ is a strong domain cover of $f^{-1}(V_0).$ Further,
$$f^{-1}(A) \cap f^{-1}(B)=f^{-1}(A \cap B)=f^{-1}(N_{2k})=M_{2k}$$ and $f \r (f^{-1}(A) \cap f^{-1}(B))=f \r M_{2k} \colon M_{2k} \to N_{2k}$ is a normal branched covering by (iii). 
Thus $V_0$ is a destructive neighbourhood of $y.$ The same argument shows that $V_k:=\mathrm{int}B_k$ is a destructive neighbourhood of $y$ for all $k \in \N.$ Thus $y$ has a neighbourhood basis consisting of neighbourhoods that are destructive with respect to $f.$
\end{proof}

Theorem \ref{45} implies that Theorems \ref{11} and \ref{eka} follow from the following result.

\begin{Thm}\label{15}Let $f \colon X \to Y$ be a proper branched covering between manifolds and let  
\begin{equation*}
\xymatrix{
& X_f \ar[dl]_p \ar[dr]^{\overline{f}}& \\
X \ar[rr]^f & & Y }
\end{equation*}
be a commutative diagram of branched coverings so that $X_f$ is a connected, locally connected Hausdorff space and $p \colon X_f \ra X$ and $q \colon X_f \ra Y$ are proper normal branched coverings. Suppose there exists a destructive point $y \in Y$ of $f$. Then $H_1(W; \Z) \neq 0$ for all open sets $W \subset X_f$ satisfying $$\overline{f}^{-1}\{y\} \cap W \neq \es.$$  
\end{Thm}

We begin the proof of Theorem \ref{15} with two lemmas. The following observation is well known for experts.

\begin{Apulause}\label{2} Let $X$ be a locally connected Hausdoff space and $W \subset X$ an open and connected subset. Suppose there exists open and connected subsets $U,V \subset W$ so that $W=U \cup V$ and $U \cap V$ is not connected. Then the first homology group $H_1(W; \Z)$ is not trivial. 
\end{Apulause}

\begin{proof} Towards contradiction we suppose that $H_1(W; \Z)=0.$ Then the reduced Mayer-Vietoris sequence
$$\to H_1(W;\Z) \to \widetilde{H}_0(U \cap V;\Z) \to \widetilde{H}_0(U; \Z) \oplus \widetilde{H}_0(V; \Z) \to \widetilde{H}_0(W; \Z)$$
takes the form
$$0 \to \widetilde{H}_0(U \cap V; \Z) \to 0.$$
Thus, $\widetilde{H}_0(U \cap V; \Z)=0.$ Thus $U \cap V$ is connected, which is a contradiction. Thus, $H_1(W; \Z)$ is not trivial. 
\end{proof}

The following lemma is the key observation in the proof of Theorem \ref{15}. 

\begin{Apulause}\label{kuu}Suppose $f \colon X \to Y$ is a branched covering between manifolds. Suppose $W$ is a connected locally connected Hausdorff space and $p \colon W \to X$ and $q \colon W \to Y$ are normal branched coverings so that the diagram
\begin{equation*}
\xymatrix{
& W \ar[ld]_{p} \ar[rd]^{q} &\\
X \ar[rr]^{f}  & & Y}
\end{equation*} 
commutes. Suppose there exists an open and connected subset $C_1 \subset Y$ so that $D_1=f^{-1}(C_1)$ is connected and $f \r D_1 \colon D_1 \to C_1$ is a normal branched covering. Then $f$ is a normal branched covering, if $E_1=q^{-1}(C_1)$ is connected. 
\end{Apulause}

\begin{proof}
Since $E_1:=q^{-1}(C_1)$ is connected, we have 
$$\mathrm{Deck}(q)=\{\tau \in \mathrm{Deck}(q) : \tau(E_1)=E_1\} \cong \mathrm{Deck}(q \r E_1 \colon E_1 \ra C_1)$$ 
and
$$\mathrm{Deck}(p)=\{\tau \in \mathrm{Deck}(p) : \tau(E_1)=E_1\} \cong \mathrm{Deck}(p \r E_1 \colon E_1 \ra D_1),$$
where the isomorphisms are canonical in the sense that they map every deck-homomorphism $\tau \colon W \to W$ to the restriction $\tau \r E_1 \colon E_1 \ra E_1.$

Since $f \r D_1 \colon D_1 \ra C_1$ is a normal branched covering, 
$$\mathrm{Deck}(p \r E_1 \colon E_1 \ra D_1) \subset \mathrm{Deck}(q \r E_1 \colon E_1 \ra C_1)$$ is a normal subgroup. Hence, $\mathrm{Deck}(p) \subset \mathrm{Deck}(q)$ is a normal subgroup. Hence, the branched covering $f \colon X \ra Y$ is normal. 
\end{proof}

\begin{proof}[Proof of Theorem \ref{15}] Let $W \subset X_f$ be a open set and $y \in f(W)$ a destructive point and $x \in \bar f^{-1}\{y\}$. By Lemma \ref{2}, to show that $H_1(W; \Z)\neq 0$ it is sufficient to show that there exists a domain cover of $W$ that is not strong. 

Let $V_0$ be a destructive neighbourhood of $y$ so that the $x$-component $W_0$ of is a normal neighbourhood of $x$ in $W.$ Let $\{A,B\}$ be a strong domain cover of $V_0$ so that $y \in \bar B \subset V_0$ and $\{W_0^A,W_0^B\}$ is a domain cover of $W_0$ for $W_0^A:=(f\r W_0)^{-1}(A)$ and $W_0^B:=(f\r W_0)^{-1}(B),$ (see Lemma \ref{fun}).  

We first show that $\{W_0^A,W_0^B\}$ is not strong. Suppose towards contradiction that $\{W_0^A,W_0^B\}$ is strong. Then $A \cap B,$ $f^{-1}(A) \cap f^{-1}(B)$ and $W^A_0 \cap W^B_0$ are connected and
\begin{equation*}\label{hei}
\xymatrix{
& W^A_0 \cap W^B_0 \ar[ld]_{p \r W^A_0 \cap W^B_0} \ar[rd]^{\overline{f}\r W^A_0 \cap W^B_0} &\\
f^{-1}(A) \cap f^{-1}(B) \ar[rr]^{f \r f^{-1}(A) \cap f^{-1}(B)}  & & A \cap B}
\end{equation*} 
is a commutative diagram of branched covers. 
In particular, since $f \r f^{-1}(A) \cap f^{-1}(B)$ is a normal branched cover $\mathrm{Deck}(p \r W^A_0 \cap W^B_0)\subset \mathrm{Deck}(\bar f \r W^A_0 \cap W^B_0)$ is a normal subgroup. 
On the other hand, since 
$$W^A_0 \cap W^B_0=(\bar f \r W_0) ^{-1}(A \cap B)=(p \r W_0)^{-1}(f^{-1}(A) \cap f^{-1}(B)) \subset W_0$$ is connected,
$\mathrm{Deck}(p \r W^A_0 \cap W^B_0)\cong \mathrm{Deck}(p \r W_0)$, $\mathrm{Deck}(\bar f \r W^A_0 \cap W^B_0) \cong \mathrm{Deck}(\bar f \r W_0)$ and in particular, $\mathrm{Deck}(p \r W_0)\subset \mathrm{Deck}(\bar f \r W_0)$ is a normal subgroup. Thus the factor $f \r f^{-1}(V_0) \colon f^{-1}(f(V_0)) \to V_0$ of $\bar f \r W_0 \colon W_0 \to V_0$ is a normal branched covering. This is a contradiction, since $V_0$ is a destructive neighbourhood of $y$ and we conclude that $W^A_0 \cap W^B_0 \subset W_0$ is not connected. 

Since $\overline{B} \subset V_0$ there exists a connected neighbourhood $W' \subset W$ of $W \setminus W_0^B$ so that $W' \cap W_0^B= \es.$ Now $W_0^A \cup W'$ is connected, since $W_0^A$ is connected and every component of $W'$ has a non-empty intersection with $W_0^A.$ Further, $W=W_0^B \cup (W_0^A \cup W')$ and $W_0^B \cap (W_0^A \cup W')=W_0^A \cap W_0^B.$ Thus $\{W_0^B,W_0^A \cup W'\}$ is a domain cover of $W$ that is not strong and by Lemma \ref{2} we conclude $H_1(W; \Z)\neq 0$.
\end{proof}

This concludes the proof of Theorem \ref{11}, and further by Corollary \ref{pimia}, the proof of Theorem  \ref{toka} in the introduction.

\bibliographystyle{abbrv}
\bibliography{Abelian_Church-Hemmingsen}

\end{document}